\documentclass[12pt,reqno]{article}
\usepackage{amsmath,amsthm,amsfonts,amssymb,amscd}
\usepackage[dvips]{graphicx}
\usepackage{psfrag}
\usepackage{textcomp}

\theoremstyle{plain}
\newtheorem{thm}{Theorem}[section]
\newtheorem{rk}[thm]{Remark}
\newtheorem{prop}[thm]{Proposition}

\newtheorem{lemma}[thm]{Lemma}
\newtheorem{defi}[thm]{Definition}

\newtheorem{maintheorem}{Theorem}
\newtheorem{claim}[thm]{Claim}

\newcommand{\re}{{\Bbb R}}
\newcommand{\nat}{{\Bbb N}}
\newcommand{\f}{{\mathcal F}}

\newcommand{\Si}{\Sigma}

\newcommand{\ce}{\left\langle}
\newcommand{\cd}{\right\rangle}

\title{Existence of periodic orbits for \\ sectional Anosov flows}
\author{A. M. L\'opez B.
        \thanks{
This work is partially supported by CAPES, Brazil.}}
\date{}

\begin{document}
\maketitle

\begin{abstract}
We prove that every sectional Anosov flow
(or, equivalently, every sectional-hyperbolic attracting set of a flow) on a compact manifold has a periodic orbit.
This extends the previous three-dimensional result obtained in \cite{bm}.
\end{abstract}


\section{Introduction}

\noindent
A well-known problem in dynamics
is to investigate the existence of periodic orbits for flows on compact manifolds.
This problem has a satisfactory solution under certain circunstancies.
In fact, every Anosov flow of a compact manifold has not only one but
infinitely many periodic orbits instead.
In this paper we shall investigate this problem not for Anosov but for the sectional Anosov flows
introduced in \cite{mor}. It is known for instance that
every sectional Anosov flow of a compact $3$-manifold has a periodic orbit, this was proved in \cite{bm}.
In the transitive case (i.e. with a dense orbit in the maximal invariant set)
it is known that the maximal invariant set consists of a homoclinic class and, therefore, the flow
has infinitely many periodic orbits \cite{ap}.
Another relevant result by Reis \cite{r}
proves the existence of infinitely many periodic orbits under certain conditions.
Our goal here is to extend \cite{bm} to the higher-dimensional setting.
More precisely, we shall prove that every sectional Anosov flow
(or, equivalently, every sectional hyperbolic attracting set of a flow) on a compact manifold has a periodic orbit.
Let us state our result in a precise way.

Consider a compact manifold $M$ of dimension $n\geq 3$
with a Riemannian structure $\|\cdot\|$ (sometimes we write {\em compact $n$-manifold} for short).
We denote by $\partial M$ the boundary of $M$.
Fix $X$ a vector field on $M$, inwardly
transverse to the boundary $\partial M$ (if nonempty) and denotes
by $X_t$ the flow of $X$, $t\in I\!\! R$.
The {\em maximal invariant} set of $X$ is defined by 
$$
M(X)= \displaystyle\bigcap_{t \geq 0} X_t(M).
$$
Notice that $M(X)=M$ in the boundaryless case $\partial M=\emptyset$.
A subset $\Lambda\subset M(X)$ is called {\em invariant} if
$X_t(\Lambda)=\Lambda$ for every $t\in I\!\! R$.
We denote by $m(L)$ the minimum norm of a linear
operator $L$, i.e., $m(L)= inf_{v \neq 0} \frac{\left\|Lv\right\|}{\left\|v\right\|}$.

\begin{defi}
\label{d2}
A compact invariant set
$\Lambda$ of $X$
is {\em partially hyperbolic}
if there is a continuous invariant
splitting $T_\Lambda M=E^s\oplus E^c$ 
such that the following properties
hold for some positive constants $C,\lambda$:

\begin{enumerate}
\item
$E^s$ is {\em contracting}, i.e., 
$\mid\mid DX_t(x) \left|_{E^s_x}\right. \mid\mid
\leq Ce^{-\lambda t}$, 
for all $x\in \Lambda$ and $t>0$.
\item
$E^s$ {\em dominates} $E^c$, i.e., 
$\frac{\mid\mid DX_t(x) \left|_{E^s_x}\right. \mid\mid}{m(DX_t(x) \left|_{E^c_x}\right. )}
\leq Ce^{-\lambda t}$, 
for all $x\in \Lambda$ and $t>0$.
\end{enumerate}

We say the central subbundle $E^c_x$ of $\Lambda$ is 
{\em sectionally expanding} if
$$dim(E^c_x) \geq 2 \quad 
\mbox{and} \quad 
\left| det(DX_t(x) \left|_{L_x}\right. ) \right| \geq C^{-1}e^{\lambda t},
\mbox{ for all } x \in \Lambda,\quad t>0 $$ 
and all two-dimensional subspace $L_x$ of $E^c_x$.
Here $det(DX_t(x) \left|_{L_x}\right. )$ denotes
the jacobian of $DX_t(x)$ along $L_x$.
\end{defi}

Recall that a {\em singularity} of a vector field is a zeroe of it.
We say that it is hyperbolic if
the eigenvalues of its linear part have non zero real part.
On the other hand, a point $p$ is {\em periodic}
if there is a minimal $T>0$ (called period) such that $X_T(p)=p$.
By a {\em periodic orbit} we mean the full orbit $\{X_t(p):t\in\mathbb{R}\}$
of a periodic point $p$.

\begin{defi}
\label{shs}
A {\em sectional hyperbolic set}
is a partially hyperbolic set whose singularities are hyperbolic and whose central subbundle is sectionally expanding.
\end{defi}

With these definitions we can state our main results.

\begin{maintheorem}
\label{thA}
Every sectional Anosov flow on a compact manifold has a periodic orbit.
\end{maintheorem}

An equivalent version of this result is as follows.

Given $\Lambda \in M$ compact, recall that $\Lambda$ is an {\em attracting} set 
if $\Lambda=\cap_{t>0}X_t(U)$ for some compact neighborhood $U$ of it, where 
this neighborhood is often called {\em isolating block}.
It is well known that the isolating block $U$ can be chosen to be
positively invariant, i.e., $X_t(U)\subset U$ for all $t>0$.
We call a sectional hyperbolic set with the above property as a 
sectional hyperbolic attracting.

Thus, the Theorem \ref{thA} is equivalent to the result below.

\begin{maintheorem}
 Every sectional hyperbolic attracting of a $C^1$ vector field on a compact manifold has a periodic orbit.
\end{maintheorem}

Theorem \ref{thA} will be proved extending the arguments in \cite{bm} to the higher dimensional setting.
Indeed, in Section \ref{sech} we provide useful definitions for Lorenz-like singularity, singular-cross sections and 
triangular maps for higher dimensional case ($n$-triangular maps for short) in the sectional hyperbolic context. 
Also, we extend some definition, lemmas and propositions necessaries for the next sections. 
In Section \ref{PO TM} we give sufficient conditions for a hyperbolic $n$-triangular map
to have a periodic point. In Section \ref{AP} we proved that hyperbolic $n$-triangular maps satisfying these hypotheses 
have a periodic point and we proved the Theorem \ref{thA}.

\section{Singular cross-sections and triangular maps in higher dimensions}
\label{sech}

\noindent
In this section, we provide a definition of Lorenz-like singularity in the 
sectional hyperbolic context. Also, in the same way of \cite{a}, \cite{am}, let us 
to define singular cross-sections for the higher dimensional case. We set certain maps defined on a finite 
disjoint union of these singular cross-sections. As in \cite{bm}, 
they are discontinuous maps still preserving the continuous foliation 
(but not necessarily constant). Thus, on compact manifolds of dimension $n \geq 3$, 
particularly to group of these maps we shall call them $n$-triangular maps.


\subsection{Lorenz-like singularities and uselful results}
\label{sll1}

Let $M$ be a compact $n$-manifold, $n \geq 3$.
Fix $X\in {\cal X}^1(M)$, inwardly 
transverse to the boundary $\partial M$. We denotes
by $X_t$ the flow of $X$, $t\in I\!\! R$.\\

In our context, the singular cross-sections are strongly associated with the 
Lorenz-like singularities. This leads to provide a Lorenz-like singularity's 
definition in our context too.
 
For this purpose, we start by presenting the standard definition of hyperbolic set 
and some preliminary definitions.

\begin{defi}
\label{hyperbolic}
A compact invariant set $\Lambda$ of $X$ is {\em
hyperbolic}
if there are a continuous tangent bundle
invariant decomposition
$T_{\Lambda}M=E^s\oplus E^X\oplus E^u$ and positive constants
$C,\lambda$ such that

\begin{itemize}
\item $E^X$ is the vector field's
direction over $\Lambda$.
\item $E^s$ is {\em contracting}, i.e.,
$
\mid\mid DX_t(x) \left|_{E^s_x}\right.\mid\mid
\leq Ce^{-\lambda t}$, 
for all $x \in \Lambda$ and $t>0$.
\item $E^u$ is {\em expanding}, i.e.,
$
\mid\mid DX_{-t}(x) \left|_{E^u_x}\right.\mid\mid
\leq Ce^{-\lambda t},
$
for all $x\in \Lambda$ and $t> 0$.
\end{itemize}
A closed orbit is hyperbolic if it is also hyperbolic, as a compact invariant set. An attractor is hyperbolic if it is also a hyperbolic
set. 
\end{defi}

It follows from the stable manifold theory \cite{hps} that if $p$ belongs to a hyperbolic set $\Lambda$, then the following sets

\begin{tabular}{lll}
$W^{ss}_X(p)$ & = & $\{x:d(X_t(x),X_t(p))\to 0, t\to \infty\}$ and\\
$W^{uu}_X(p)$ & = & $\{x:d(X_t(x),X_t(p))\to 0, t\to -\infty\}$ \\
\end{tabular}\\
are $C^1$ immersed submanifolds of $M$ which are tangent at $p$ to the subspaces $E^s_p$ and $E^u_p$ of $T_pM$ respectively.
Similarly,

\begin{tabular}{lll}
$W^{s}_X(p)$ & = & $ \bigcup_{t\in I\!\! R}W^{ss}_X(X_t(p))$ and\\
$W^{u}_X(p)$ & = & $ \bigcup_{t\in I\!\! R}W^{uu}_X(X_t(p))$ \\
\end{tabular}\\
are also $C^1$ immersed submanifolds tangent to $E^s_p\oplus E^X_p$ and $E^X_p\oplus E^u_p$ at $p$ respectively.
Moreover, for every $\epsilon>0$ we have that

\begin{tabular}{lll}
$W^{ss}_X(p,\epsilon)$ & = & $\{x:d(X_t(x),X_t(p))\leq\epsilon, \forall t\geq 0\}$ and\\
$W^{uu}_X(p,\epsilon)$ & = & $\{x:d(X_t(x),X_t(p))\leq \epsilon, \forall t\leq 0\}$\\
\end{tabular}\\
are closed neighborhoods of $p$ in $W^{ss}_X(p)$ and $W^{uu}_X(p)$ respectively.\\

There is also a stable manifold theorem in the case when $\Lambda$ is sectional hyperbolic set.
Indeed, denoting by $T_{\Lambda}M=E^s_{\Lambda}\oplus E^c_{\Lambda}$ the corresponding the sectional hyperbolic
splitting over $\Lambda$ we have from \cite{hps} that
the contracting subbundle $E^s_{\Lambda}$
can be extended to a contracting subbundle $E^s_U$ in $M$. Moreover, such 
an extension is tangent to a continuous foliation denoted 
by $W^{ss}$ (or $W^{ss}_X$ to indicate dependence on $X$).
By adding the flow direction to $W^{ss}$ we obtain a continuous foliation 
$W^s$ (or $W^s_X$) now tangent to $E^s_U\oplus E^X_U$.
Unlike the Anosov case $W^s$ may have singularities, all of which being
the leaves $W^{ss}(\sigma)$ passing through the singularities $\sigma$ of $X$.
Note that $W^s$ is transverse to $\partial M$
because it contains the flow direction (which is transverse to $\partial M$
by definition). So, note the following remark
\begin{equation}
\begin{tabular}{l}
It turns out that every singularity $\sigma$ of a sectional hyperbolic set $\Lambda$ \\
satisfies $W^{ss}_X(\sigma)\subset W^s_X(\sigma)$.\\
Furthermore, there are two possibilities for such a singularity, namely,  \\
either $dim(W^{ss}_X(\sigma))=dim(W^s_X(\sigma))$ (and so $W^{ss}_X(\sigma)=W^s_X(\sigma)$) \\
or $dim(W^{s}_X(\sigma))=dim(W^{ss}_X(\sigma))+1$.\\
\end{tabular}
\label{sll}
\end{equation}
In the later case we call it Lorenz-like according to the following definition. 

\begin{defi}
\label{ll} 
We say that a singularity $\sigma$ of a sectional-Anosov flow $X$ is {\em Lorenz-like}
if
$dim (W^s(\sigma))=dim (W^{ss}(\sigma))+1.$
\end{defi}

Considering the above definitions, we show the following lemmas 
presenting an elementary but very useful dichotomy for the singularities
of a sectional hyperbolic sets.

First, we consider the called hyperbolic lemma.

\begin{lemma}
\label{hl}
Let $\Lambda$ be a sectional hyperbolic set of a $C^1$ vector field $X$ of $M$.
Then, there is a neighborhood $ {\cal U} \subset {\cal X}^1(M)$ of $X$ 
and a neighborhood $U \subset M$ of $\Lambda$ such that if $Y \in {\cal U}$,
every nonempty, compact, non singular, invariant set $H$
of $Y$ in $U$ is hyperbolic {\em saddle-type} (i.e. $E^s\neq 0$ and $E^u\neq 0$).
\end{lemma}
\begin{proof}
See (\cite{mpp2}).
\end{proof}

This following result examinating the sectional hyperbolic splitting 
$T_{\Lambda}M = E^s_{\Lambda} \oplus E^c_{\Lambda}$ of a sectional Anosov 
flow $X \in \mathcal{X}^1(M)$ appears in \cite{lec}. 

\begin{thm}
\label{t1}
Let $X$ be a sectional-Anosov flow $C^1$ for $M$.
Then, every $\sigma\in Sing(X)\cap M(X)$ satisfies
$$M(X)\cap W^{ss}_X(\sigma)=\{\sigma\}.$$
\end{thm}

The following results appear in \cite{bm}, but this version is a modification  
for sectional hyperbolic sets in the higher dimensional case. 

\begin{thm} 
\label{tom}
Let $X$ be a sectional Anosov flow of $M$. Let $\sigma$ be a singularity of $M(X)$. 
If there is $x \in M(X) \setminus W^{ss}(\sigma)$ such that $\sigma \in \omega_X(x)$, 
then $\sigma$ is Lorenz-like and satisfies
$$M(X)\cap W^{ss}_X(\sigma)=\{\sigma\}.$$
\end{thm}

\begin{proof}
The equality follows from Theorem \ref{t1}. We assume that $M(X)$ is connected for, otherwise, we consider
the connected components. Suppose that $\sigma \in M(X) \cap Sing(X)$ satisfies
$\sigma \in \omega_X(x)$for some $x \in M(X) \setminus W^{ss}(\sigma)$. Let us
prove that $\sigma$ is Lorenz-like. Since $M(X)$ is maximal invariant of $X$ 
we have $\omega_X(x) \in M(X)$ and so $\sigma \in M(X)$.
Assume by contradiction that $\sigma$ is not Lorenz-like. Then, by (\ref{sll}) we have that 
$dim(W^{ss}(\sigma)) = dim(W^{s}(\sigma))$ and so $W^{ss}(\sigma) = W^{s}(\sigma)$.
Since $x \notin W^{ss}(\sigma)$, one has 
$\omega_X(x) \cap (W^s(\sigma) \setminus \left\{\sigma\right\}) \neq \emptyset$.
But recall that $\omega_X(x) \in M(X)$ and beside with $W^{ss}(\sigma) = W^{s}(\sigma)$,
we obtain that $$M(X)\cap (W^{ss}_X(\sigma) \setminus \{\sigma\}) \neq \emptyset$$
contradicting the equality in Theorem \ref{t1}. This proves the result.
\end{proof}

\begin{prop}
\label{prom}
Let $X$ be a sectional Anosov flow of $M$. If $M(X)$ has no Lorenz-like singularities, 
then $M(X)$ has a periodic orbit. 
\end{prop}

\begin{proof}
Let $x$ be a point in  $M(X) \setminus Sing(X)$. 
We claim that $\omega_X(x)$ has no singularities. Indeed,
suppose by contradiction that $\omega_X(x)$ has a singularity $\sigma$. 
By hypothesis $M(X)$ has no Lorenz-like singularities and so $\sigma$ is not Lorenz-like too. 
Hence $W^{ss}(\sigma) = W^{s}(\sigma)$ by (\ref{sll}), and by using the Theorem \ref{t1} one has 
$$M(X)\cap (W^{s}_X(\sigma) \setminus \{\sigma\}) = \emptyset$$
It follows in particular that  $(W^{s}_X(\sigma) \setminus \{\sigma\}) $
is no belongs to $M(X)$. Since $x \in (M(X)\setminus Sing(X))$ we conclude that 
$x \notin W^{s}_X(\sigma)=W^{ss}_X(\sigma)$. It follows from Theorem \ref{tom} that 
$\sigma$ is Lorenz-like. This is a contradiction and the claim follows.

Now we conclude the proof of the proposition. Clearly $\omega_X(x) \subset M(X)$ since $M(X)$ is
compact. The claim and \ref{hl} imply that $\omega_X(x)$ is a hyperbolic set. It follows from
the Shadowing Lemma for flows \cite{haka} that there is a periodic orbit of $X_t$ close to
$\omega_X(x)$. Since $M(X)$ is the maximal invariant, we have that such a periodic orbit is 
contained in $M(X)$. Then we obtain the result.
\end{proof}

\subsection{Singular cross-section in higher dimension}
\label{scs}

Here, we will define singular cross-section in the higher
dimensional context. First, we will denote a cross-section by $\Sigma$ and 
its boundary by $ \partial \Sigma$. 
Also, the hypercube $I^k=\left[-1,1\right]^k$ will be 
submanifold of dimension $k$ with $k \in \nat$. 

Let $\sigma$ be a Lorenz-like singularity. Hereafter, 
we will denote $dim(W^{ss}_X(\sigma))=s$, $dim(W^{u}_X(\sigma))=u$ and 
therefore $dim(W^{s}_X(\sigma))=s+1$ by definition. 
Moreover $W^{ss}_X(\sigma)$ separates $W^s_{loc}(\sigma)$
in two connected components denoted by $W^{s,t}_{loc}(\sigma)$ 
and $W^{s,b}_{loc}(\sigma)$ respectively. 

Thus, we begin by considering $B^u[0,1] \approx I^u$ and $B^{ss}[0,1] \approx I^s$ where 
$B^{ss}[0,1]$ is the ball centered at zero and radius $1$ contained in $\re^{dim(W^{ss}(\sigma))}=\re^{s}$ and 
$B^{u}[0,1]$ is the ball centered at zero and radius $1$ contained in $\re^{dim(W^{u}(\sigma))}=\re^{n-s-1}=\re^{u}$.

\begin{defi}
\label{n1}
A {\em singular cross-section} of a Lorenz-like singularity 
$\sigma$ consists of a pair of submanifolds $\Sigma^t, \Sigma^b$, where 
$\Sigma^t, \Sigma^b$ are cross-sections such that \\

\begin{tabular}{c}
$\Sigma^t$ is transversal to $W^{s,t}_{loc}(\sigma)$ and 
$\Sigma^b$ is transversal to $W^{s,b}_{loc}(\sigma)$.\\ 
\end{tabular}\\

\end{defi} 

Note that every singular cross-section
contains a pair singular submanifolds $l^t,l^b$ 
defined as the intersection of the local stable manifold of $\sigma$ with $\Sigma^t,\Sigma^b$ 
respectively and, additionally, $dim(l^*)=dim(W^{ss}(\sigma))$ $(*=t,b)$.\\

Thus a singular cross-section $\Sigma^*$ will be a {\em hypercube of dimension $(n-1)$},   
i.e., diffeomorphic to $B^u[0,1] \times B^{ss}[0,1]$.  
Let $f: B^u[0,1] \times B^{ss}[0,1] \longrightarrow \Sigma^*$ be the diffeomorphism, 
such that $$f(\left\{0\right\} \times B^{ss}[0,1])=l^*$$ and $\left\{0\right\}=0 \in \re^u$.
Define
$$\partial \Sigma^* = \partial^h \Sigma^* \cup \partial^v \Sigma^*$$ by:

\begin{tabular}{l}
$\partial^h \Sigma^* = \left\{\right.$union of the boundary submanifolds
 which are transverse to $l^*$ $\left.\right\}$ and\\
$\partial^v \Sigma^* = \left\{\right.$union of the
boundary submanifolds which are parallel to $l^*$ $\left.\right\}.$
\end{tabular}\\

From this decomposition we obtain that 
$$\partial^h \Sigma^*  =  (I^u \times [\cup_{j=0}^{s-1} (I^j \times \left\{-1\right\} \times I^{s-j-1})])
\bigcup (I^u \times[\cup_{j=0}^{s-1} (I^j \times \left\{1\right\} \times I^{s-j-1})])$$
$$\partial^v \Sigma^*  =  ([\cup_{j=0}^{u-1} (I^j \times \left\{-1\right\} \times I^{u-j-1})] \times I^s)
\bigcup ([\cup_{j=0}^{u-1} (I^j \times \left\{1\right\} \times I^{u-j-1})] \times I^s),$$
where $I^0 \times I=I$.\\

Hereafter we denote $\Sigma^* = B^u[0,1] \times B^{ss}[0,1]$.


\subsubsection{Refinement of singular cross-sections and induced foliation}

Hereafter, $X$ denotes a sectional Anosov flow of a compact $n$-manifold $M$, $n \geq 3$, 
$X \in \mathcal{X^1(M)}$. Let $M(X)$ be the maximal invariant of $X$.

In the same way of \cite{bm}, we obtain an induced foliation $\f$ on $\Si$ 
by projecting $\f^{ss}$ onto $\Si$, where $\f^{ss}$ denotes the invariant continuous 
contracting foliation on a neighborhood of $M(X)$ \cite{hps}.   

For the refinement, since $\Sigma^*= B^u[0,1] \times B^{ss}[0,1]$, 
we will set up a family of singular cross-sections as follows: 
Given $0<\Delta\leq 1$ small, we define $\Sigma^{*,\Delta}=B^u[0,\Delta] \times B^{ss}[0,1]$,
such that
$$l^* \subset \Sigma^{*,\Delta}\subset \Sigma^* ,\,\,\, i.e$$
$$(l^*=\{0\}\times B^{ss}[0,1]) \subset
(\Sigma^{*,\Delta}=B^u[0,\Delta] \times B^{ss}[0,1]) \subset 
(\Sigma^*= B^u[0,1] \times B^{ss}[0,1]),$$
where we fix a coordinate system $(x^*,y^*)$
in $\Sigma^*$ ($*=t,b$). We will assume that $\Sigma^*=\Sigma^{*,1}$.


\subsection{$n$-Triangular maps}

We begin by reminding the three-dimensional case \cite{bm}, where the authors choose the 
cross-sections as copies of $\left[0,1\right] \times \left[0,1\right]$ and they define 
maps on a finite disjoint union of these one copies called {\em triangular maps}. This concept 
is frequently used by maps on $\left[0,1\right] \times \left[0,1\right]$ preserving the 
constant vertical foliation. Also, they assume two hypotheses imposing certain amount of 
differentiability close to the points whose iterates fall eventually in the interior of $\Sigma$.

For the higher dimensional case, we will define a certain maps on a finite 
disjoint union of singular cross-sections $\Sigma$. Here, we will modify the triangular map's definition and we 
will impose some suitable properties in order to denine a {\em triangular hyperbolic map}. 
These maps could be discontinuous and they will preserve still the continuous foliation 
(but not necessarily constant). Thus, on compact manifolds of dimension $n \geq 3$, 
particularly to group of these maps we shall call them $n$-triangular maps.

By using the above definitions about singular cross-sections, 
we will provide the following definitions.

\begin{defi}
Let $\Sigma$ be a disjoint union of finite singular cross-sections $\Sigma_i$, $i=1,\dots,k.$. 
We denote by $l_{0_i}$ to the singular leaf of the singular cross-section $\Sigma_i$.
Here $L_0$ stands the union of singular leaves, i.e., $$L_0=\bigcup_{i=1}^n l_{0_i}.$$
Recall that $\partial^v \Sigma_i$ is the union of the
boundary submanifolds which are parallel to $l_{0_i}$. In the same way we set by $\partial^v \Sigma$ as 
$$ \partial^v \Sigma = \bigcup_{i=1}^n \partial^v \Sigma_i$$
\end{defi}

Hereafter, given a map $F$ we will denote its domain by $Dom(F)$. 

\begin{defi}
Let $F: Dom(F) \subset \Sigma \rightarrow \Sigma$ be a map $x$ a point in $Dom(F)$. We say that 
$x$ is a {\em periodic point} of $F$ if there is $n \geq 1$ such that $F^j(x) \in Dom(F)$ 
for all $0\leq j \leq n-1$ and $F^n(x) = x$
\end{defi}

\begin{defi}
We say that a submanifold $c$ of $\Sigma$ is a {\em $k$-surface} if it is the image 
of a $C^1$ injective map $c : Dom(c) \subset \re^k \rightarrow \Sigma$, with $Dom(c)$ 
being $I^k$ and $k \leq n-1$. For simplicity, hereafter $c$ stands the image of this one map. 
A $k$-surface $c$ is {\em vertical} if it is the graph of a $C^1$ map $g: I^{n-k-1}\rightarrow I^{k}$, 
i.e., $c=\left\{(g(y),y): y \in I^k\right\} \subset \Sigma$.   
\end{defi}

\begin{defi}
A continuous foliation $\f_i$ on a component $\Si_i$ of $\Si$ is called 
{\em vertical} if its leaves are vertical $s$-surfaces and $\partial ^v \Si \subset \f_i$, 
where $s=dim (B^{ss}[0,1])$. A {\em vertical foliation} $\f$ of $\Si$ is a foliation which restricted to each component $\Si_i$ 
of $\Si$ is a vertical foliation.
\end{defi}

It follows from the definition above that the leaves $L$ of
a vertical foliation $\f$ are vertical $s$-surfaces hence differentiable ones.
In particular, the tangent space $T_xL$ is
well defined for all $x\in L$. 

\begin{rk}
Note that, given a singular cross-section $\Si$ equipped with a vertical foliation $\f$, 
one has that $dim (\f) = dim (B^{ss}[0,1])=s$, and each leaf $L$ of 
$\f$ has the same dimension of $W^{ss}(\sigma)$, being $\sigma$ the Lorenz-like singularity 
associated to $\Si$.
\end{rk}

\begin{defi}
Let $F: Dom(F)\subset \Sigma \to \Sigma$ a map and $\f$ be a vertical foliation
on $\Sigma$. We say that $F$ {\em preserves} $\f$ if
for every leaf $L$ of $\f$ contained in $Dom(F)$ there is a leaf $f(L)$ of $\f$ such that
$F(L)\subset f(L)$ and the restricted map $F/_L:L \to f(L)$ is continuous.
\end{defi}

If $\f$ is a vertical foliation on $\Sigma$ a subset $B\subset \Sigma$
is a {\em saturated set} for $\f$ if it is union of leaves of $\f$.
We shall write $\f$-saturated for short.

\begin{defi}[{\bf $n$-Triangular map}]
\label{tim}
A map $F:Dom(F)\subset \Sigma\to \Sigma$ is called {\em $n$-triangular}\index{Map!Triangular} if
it preserves a vertical foliation ${\mathcal F}$
on $\Sigma$ such that $Dom(F)$ is $\mathcal F$-saturated and $dim(\Si)=n-1$, with $n \geq 3$. 
Note that a $3$-triangular map is the classical triangular map.
\end{defi}

\subsection{Hyperbolic $n$-triangular maps}

In the same way of \cite{bm}, in order to find periodic points for 
$n$-triagular maps, also we introduce some kind of hyperbolicity for these maps. 
The hyperbolicity will be defined through cone fields in $\Sigma$:
We denote by $T\Sigma$ the tangent bundle of $\Sigma$. 
Given $x\in \Sigma$, $\alpha>0$ and a linear subspace $V_x \subset T_x \Sigma$ we denote
by $C_\alpha(x,V_x) \equiv C_\alpha(x)$ the cone
around $V_x$ in $T_x\Sigma$ with inclination $\alpha$, namely
$$
C_\alpha(x)= \{v_x\in T_x\Sigma:\angle(v_x,V_x)\leq \alpha\}.
$$
Here $\angle(v_x,V_x)$ denotes the angle between a vector $v_x$ and the subspace $V_x$.
A {\em cone field} in $\Sigma$ is a continuous map
$C_\alpha:x\in \Sigma\to C_\alpha(x) \subset T_x \Sigma$,
where $C_\alpha(x)$ is a cone
with constant inclination $\alpha$ on $T_x\Sigma$.
A cone field $C_\alpha$ is called
{\em transversal}\index{Cone!Field!Transverse}
to a vertical foliation ${\mathcal F}$ on $\Sigma$ if
$T_xL$ is not contained in $C_\alpha(x)$
for all $x \in L$ and all $L \in {\mathcal F}$.

Now we can define hyperbolic $n$-triangular map.

\begin{defi}[{\bf Hyperbolic $n$-triangular map}]
\label{lambda-hyperbolic}
Let $F:Dom(F)\subset \Sigma\to\Sigma$ be a $n$-triangular map with associated
vertical foliation ${\mathcal F}$.
Given $\lambda > 0$ we say that
$F$ is {\em $\lambda$-hyperbolic}\index{Map!Triangular!$\lambda$-hyperbolic}
if there is a cone field $C_\alpha$ in $\Sigma$ such that:

\begin{enumerate}
\item
$C_\alpha$ is transversal to ${\mathcal F}$.
\item
If $x\in Dom(F)$ and $F$ is differentiable at $x$,
then
$$
DF(x)(C_\alpha(x))\subset Int(C_{\alpha/2}(F(x)))
$$
and
$$
\mid\mid DF(x)\cdot v_x\mid\mid \geq \lambda\cdot \mid\mid v_x\mid\mid,
$$
for all $v_x\in C_\alpha(x)$.
\end{enumerate}
\end{defi}

\section{Periodic points for hyperbolic $n$-triangular maps}
\label{PO TM}

\noindent

In the three dimensional case \cite{bm}, the authors impose certain conditions or properties so-called 
hypotheses (H1) and (H2) on triangular maps. With this conditions the periodic point arose on the triangular map. 
Therefore, the general tools for searching our definitions is trying and reproducing these hypotheses,
in the higher dimensional setting. Here, the topology turns out to play a significant role in this extension,
imposing certain restrictions on the manifolds, kind of foliations and singular cross-sections one may have. 

In this section we give sufficient conditions for a hyperbolic $n$-triangular map
to have a periodic point. 

\subsection{Hypotheses {\bf (A1)-(A2)}}
\label{H}
They impose some regularity around those leaves whose iterates {\em eventually fall into}
$\Sigma\setminus(\partial ^v \Si)$.
To state them we will need the following definition.
If ${\mathcal F}$ is foliation we use the notation $L\in {\mathcal F}$ to mean that
$L$ is a leaf of ${\mathcal F}$.

\begin{defi}
\label{n(L)}
Let $F:Dom(F)\subset \Sigma\to \Sigma$ be a triangular map such that $\partial ^v \Si \subset Dom(F)$.
For all $L\in {\mathcal F}$ contained in $Dom(F)$ we define the (possibly $\infty$) number
$n(L)$ as follows:

\begin{enumerate}
\item
If $F(L)\subset \Sigma\setminus(\partial ^v \Si)$ we define $n(L)=0$.
\item
If $F(L)\subset \partial ^v \Si$ we define
$$
n(L)=\sup\{n\geq 1:F^i(L)\subset Dom(F) \mbox{ and }
$$
$$
F^{i+1}(L)\subset \partial ^v \Si, \forall 0\leq i\leq n-1\}.
$$
\end{enumerate}
\end{defi}

Essentially $n(L)+1$ gives the first non-negative iterate of $L$
falling into $\Sigma\setminus (\partial ^v \Si)$.

\begin{rk}
\label{vl}
On the other hand, note that if $n(L_*) \geq 1$ for $L_* \in \f$ contained in $Dom(F)$, 
for every neighborhood $S$ of $L_*$, $F(S) \cap \partial ^v \Si \neq \emptyset$. 
We denote $$V_{L_*}(S) \equiv V_{L_*} = F(S) \cap \partial ^v \Si.$$ 
Therefore $V_{L_*} \neq \emptyset$ and if $L_* \notin \partial ^v \Si $, 
$V_{L_*}$ splits $F(S)$ in two connected components $S'_1,S'_2$. 
It shows that there exists three connected components $S_0,S_1, S_2$ of $S$ 
such that 

\begin{center}
\begin{tabular}{l}
$S= S_0 \cup S_1 \cup S_2$ and  \\
$F(S_0) = V_{L_*}$, \quad $F(S_1) = S'_1$,  \quad and  $F(S_2) = S'_2$ .\\
\end{tabular}
\end{center}
\end{rk}

Given $L \in \f$ contained in $Dom(F)$, the number $n(L)$ and the 
neighborhood $V_{L_*}$ play fundamental role in the following definition.

\begin{defi}[{\bf Hypotheses (A1)-(A2)}]
\label{hypocampo}
Let $F\colon Dom(F)\subset \Sigma\to \Sigma$ be
a $n$-triangular map such that $\partial ^v \Si \subset Dom(F)$. We say that $F$ satisfies:

\begin{description}
\item[(A1)]
If $L\in {\mathcal F}$ satisfies $L\subset Dom(F)$
and $n(L)=0$, then there is a ${\mathcal F}$-saturated neighborhood $S$
of $L$ in $\Sigma$ such that the restricted map $F/_S$ is $C^1$.

\item[(A2)]
If $L_*\in {\mathcal F}$ satisfies $L_*\subset Dom(F)$, $1\leq n(L_*)<\infty$ and
$$ F^{n(L_*)}(L_*)\subset Dom(F),$$
then there is a connected neighborhood $S \subset Dom(F)$ of $L_{*}$ such that 
$S=S_0\cup S_1 \cup S_2$ (see Remark (\ref{vl})) and the connected
components $S_1,S_2$ (possibly equal if $L_*\subset \partial ^v \Si$) beside $V_{L_*}$ 
satisfying the properties below:

\begin{enumerate}
\item
Both $F(S_1)$ and $F(S_2)$ are contained
in $\Sigma \setminus (\partial ^v \Si)$.

\item
\label{fvl}
$F^i(V_{L_*}) \subset \partial ^v \Si$ for all $0 \leq i \leq n(L_*)$, i.e.,
as $V_{L_*}$ is $\f$-saturated, for any $L' \in V_{L_*}$ one has that 
$n(L')=n(L_*)-1$ (or $n(L')=n(L_*)$ for the case $L_* \in \partial ^v \Si$)  

\item
For each $j\in \{0,1,2\}$, there is $0 \leq n^j(L_*) \leq n(L_*)+1$ 
such that if $y_l\in S_{j}$ is a sequence converging to $y \in L_{*}$,
then $F(y_l)$ is a sequence converging to $F^{n^j(L_*)}(y)$.
If $n^j(L_*)=0$, then $F$ is $C^1$ in $S_j \cup V_{L_*}$. 
Note that $n^0(L_*) = n(L_*)$ by \ref{fvl}.

\item
If $L_* \subset \Sigma \setminus (\partial ^v \Si)$
(and so $S_1\neq S_2$), then
either $n^1(L_*)=0$ and $n^2(L_*) \geq 1$ or
$n^1(L_*) \geq 1$ and $n^2(L_*)=0$.  

\end{enumerate}
\end{description}
\end{defi}

\subsection{Existence on hyperbolic $n$-triangular maps}

The following theorem will deal with the existence of periodic
points for hyperbolic $n$-triangular maps satisfying (A1) and (A2). 
Since the conditions (A1) and (A2) generalize the conditions in \cite{bm}, 
also we have that the three dimensional Lorenz attractor return map is an 
example for us. Recall the Lorenz attractor return map has a periodic point 
and it is a $\lambda$-hyperbolic $n$-triangular map satisfying (A1) and (A2) 
with $\lambda$ large and $Dom(F) = \Si \setminus L_0$. 
Indeed, the main motivation is show the following theorem for higher dimensional case. 
More precisely, 

\begin{thm}
\label{op2}
Let $F$ be a $\lambda$-hyperbolic $n$-triangular map satisfying (A1) and (A2) with 
$\lambda > 2$ and $Dom(F) = \Si \setminus L_0$. Then, $F$ has a periodic point.
\end{thm}

Since the proof of Theorem \ref{op2} is technical, 
we include some preliminaries for the proof and the 
proof in the next section.

\section{Existence of the periodic point}
\label{AP}

In this section we shall prove the Theorem \ref{op2}. The proof follow the same way of 
\cite{bm}. We will extend and modify some results for the higher dimensional case. 
In Subsection \ref{to1} we present preliminary lemmas for the proof of the Theorem \ref{op2}. 
In Subsection \ref{to2} we prove the theorem.

\subsection{Preliminary lemmas}
\label{to1}
Hereafter we fix $\Sigma$ as in Subsection \ref{scs}.
Let $k$ be the number of components of $\Sigma$.
We shall denote by $SL$ the leaf space of a
vertical foliation ${\mathcal F}$ on $\Sigma$.
It turns out that $SL$ is a disjoint union of $k$-copies 
$B^{u}_1\left[0,1\right],\cdots ,B^{u}_k\left[0,1\right]$ of $B^{u}\left[0,1\right]$. 
We denote by ${\mathcal F}_B$ the union of all leaves of ${\mathcal F}$ intersecting $B$. 
If $B=\{x\}$, then ${\mathcal F}_x$ is the leaf of ${\mathcal F}$ containing $x$. 
If $S,B \subset \Sigma$ we say that $S$ {\em cover} $B$ whenever
$B \subset {\mathcal F}_S$.

The lemma below quotes some elementary properties
of $n(L)$ in Definition \ref{n(L)}.

\begin{lemma}
\label{properties}
Let $F\colon Dom(F)\subset \Sigma\to \Sigma$ be a triangular map with associated
vertical foliation ${\mathcal F}$. If $L\in {\mathcal F}$ and $L\subset Dom(F)$, then:

\begin{enumerate}
\item
If {\em $F$ has no periodic points} and $\partial^v \Sigma \subset Dom(F)$,
then $$n(L)\leq 2k.$$

\item
$n(L)=0$ if and only if $F(L)\subset \Sigma \setminus (\partial^v \Sigma)$.

\item
$F^i(L)\subset \partial^v \Sigma$ for all $1\leq i\leq n(L)$.

\item
If $F^{n(L)}(L)\subset Dom(F)$, then 
$F^{n(L)+1}(L)\subset \Sigma \setminus (\partial^v \Sigma)$.

\end{enumerate}
\end{lemma}

If $F:Dom(F)\subset \Sigma \to \Sigma$ is a $n$-triangular
map with associated foliation ${\mathcal F}$,
then we also have an associated $u$-dimensional map
$$f:Dom(f)\subset SL \to SL.$$
This map allows us to obtain certain geometric properties 
for the singular cross-section whole.

We use this map in the definition below.

\begin{defi}
\label{limite-f}
Let $F:Dom(F)\subset \Sigma \to \Sigma$ a triangular
map with associated foliation
${\mathcal F}$ and
$f:Dom(f)\subset SL\to SL$ its associated
$u$-dimensional map.
Then we define the following limit sets:

$${\mathcal V}=
\{f(B):B\in {\mathcal F},B\subset Dom(F)
\mbox{ and } B\subset \partial ^v \Si \}.$$
$${\mathcal L}= \bigcup\left\{K_i:  i\in\{1,...,k\}, \lim_{L\to L_{0i}}f(L) \mbox{ exists} \text{ and such that}
K_i=\lim_{L\to L_{0i}}f(L)\right\}.$$

\end{defi}

The lemma below is a direct consequence of {\bf (A2)}.

\begin{lemma}
\label{remarkD(F)}
Let $F\colon Dom(F)\subset \Sigma\to \Sigma$ a $n$-triangular map satisfying {\bf (A2)}
and ${\mathcal F}$ be its associated foliation.
Let $L_*$ be a leaf of $\f$, $L_*\subset Dom(F)$, $1\leq n(L_*)<\infty$ and
$F^{n(L_*)}(L_*)\subset Dom(F)$. If there is a sequence $(L_k)_{k \in \nat}$ 
such that $L_k \to L_{*}$, then:

\begin{enumerate}
	\item[(1)]
	If $\#\{L: L \in (V_{L_{*}} \cap (F(L_k))_{k \in \nat })\} = \infty$, 
	then there exists a subsequence $(L_{k_l})_{l \in \nat}$ 
	such that $Lim_{l \to \infty} F(L_{k_l}) = F(L_*)$.
	\item[(2)]
	If $\#\{L: L \in ( S'_1 \cap (F(L_k))_{k \in \nat })\} = \infty$, 
	then there exists a subsequence $(L_{k_l})_{l \in \nat}$ 
	such that $Lim_{l \to \infty} F(L_{k_l}) = F^{n_1(L_*)}(L_*)$.
	\item[(3)]
	If $\#\{L: L \in ( S'_2 \cap (F(L_k))_{k \in \nat })\} = \infty$, 
	then there exists a subsequence $(L_{k_l})_{l \in \nat}$ 
	such that $Lim_{l \to \infty} F(L_{k_l}) = F^{n_2(L_*)}(L_*)$.
\end{enumerate}

In each case the corresponding limits
belong to $$\partial ^v \Si \cup {\mathcal V}.$$

\end{lemma}

\begin{proof}
By hypotheses $F$ satisfies the property $(A2)$, so there is a connected 
neighborhood $S=S_0 \cup S_1 \cup S_2$ of $L_*$. 
If we have (1), since $L_k \to L_{*}$, one has that 
$\#\{L: L \in ( S_0 \cap (L_k)_{k \in \nat })\} = \infty$. Then, 
there exists a subsequence $(L_{k_l})_{l \in \nat} \subset S_0$ 
such that $L_{k_l} \to L_{*}$. Thus, $Lim_{l \to \infty} F(L_{k_l}) = F(L_*)$.

If we have (2) beside the (A2) property, in the same way one has that 
there exists a subsequence $(L_{k_l})_{l \in \nat} \subset S_1$ 
such that $L_{k_l} \to L_{*}$ and $Lim_{l \to \infty} F(L_{k_l}) = F^{n_1(L_*)}(L_*)$.
Analogously for (3).
\end{proof}

Given a map $F\colon Dom(F)\subset \Sigma\to \Sigma$
we define its {\em discontinuity
set} $D(F)$ by

\begin{equation}
\label{D(F)'}
D(F)=\{x\in Dom(F): F\ \mbox{ is discontinuous in }\ x\}.
\end{equation}

In the sequel we derive useful properties of $Dom(F)$ and $D(F)$.

\begin{lemma}
\label{D(F)1}
Let $F$ be a $n$-triangular map satisfying {\bf (A1)}, 
$F:Dom(F) \subset \Sigma \to \Sigma$ and ${\mathcal F}$ be 
its associated foliation. If $L \in {\mathcal F}$ and
$L \subset D(F)$, then $F(L) \subset \partial ^v \Si$.
\end{lemma}

\begin{proof}
Suppose by contradiction that $L \subset D(F)$ and
$F(L) \subset \Sigma \setminus (\partial ^v \Si)$.
These properties are equivalent to
$n(L)=0$ by Lemma \ref{properties}-(2).
Then, by using {\bf (A1)}, there is a neighborhood
of $L$ in $\Sigma$ restricted to which $F$ is $C^1$.
In particular, $F$ would be continuous in $L$
which is absurd.
\end{proof}

\begin{lemma}
\label{D(F)}
Let $F$ be a $n$-triangular map satisfying {\bf (A1)-(A2)}, 
$F:Dom(F) \subset \Sigma \to \Sigma$ and ${\mathcal F}$ be 
its associated foliation. If $F$ has no periodic points and 
$\partial ^v \Si \subset Dom(F)$, then
$Dom(F) \setminus D(F)$ is $\mathcal F$-saturated,
open in $Dom(F)$ and
$F/_{(Dom(F)\setminus D(F))}$ is $C^1$.
\end{lemma}

\begin{proof}
In order to prove the lemma, it suffices to show that $\forall x \in Dom(F) \setminus D(F)$ 
there is a neighborhood $S$ of ${\mathcal F}_x$ in $\Sigma$ such that $F/_S$ is $C^1$.
To find $S$ we proceed as follows.
Fix $x\in Dom(F) \setminus D(F)$. As $Dom(F)$ is $\mathcal F$-saturated, 
one has ${\mathcal F}_x \subset Dom(F)$ and so $n({\mathcal F}_x)$ is
well defined.
By using the Lemma \ref{properties}-(1), one has 
$$n({\mathcal F}_x)<\infty.$$

\medskip
If $n({\mathcal F}_x)=0$, then the neighborhood $S$ of $L={\mathcal F}_x$
in {\bf (A1)} works.

\medskip 
By simplicity, if $n({\mathcal F}_x)\geq 1$ let us denote $L_*={\mathcal F}_x$.
Clearly $1\leq n(L_*)<\infty$ and Definition \ref{n(L)} of $n(L_*)$ implies 
$f^{n(L_*)}(L_*) \subset \partial^v \Si$.
By hypothesis $\partial^v \Si \subset Dom(F)$ and then
$$F^{n(L_*)}(L_*) \subset Dom(F).$$
So, we can choose $S$ as the neighborhood of $L_*$ in {\bf (A2)}.
Let us prove that this neighborhood works.

First we claim that $L_{*} \subset \partial ^v \Si$.
Indeed, if $L_{*} \subset \Sigma \setminus( \partial ^v \Si)$, then $S$ 
has three different connected components $S_0,S_1,S_2$.
By {\bf (A2)}-(4) we can assume $n^1(L_*)>1$ where $n_1(L)$ comes from
{\bf (A2)}-(3).
Choose sequence $x_i^1\in S_1\to x$ then $F(x_i^1) \to F^{n_1(L_*)}(x)$ by {\bf (H2)}-(3).
As $F$ is continuous in $x$ we also have $F(x_i^1) \to F(x)$ and then
$F^{n^1(L_*)}(x)=F(x)$ because limits are unique.
Thus, $F^{n^1(L_*)-1}(x)=x$ because $F$ is injective and so $x$ is a 
periodic point of $F$ since $n^1(L_*)-1\geq 1$.
This contradicts the non-existence of periodic
points for $F$. The claim is proved.

The claim implies that $S$ has a two components, i.e, $S_0$ and $S_1=S_2$.
By using {\bf (A2)}, for the component $S_0$ one has $n^0(L_*)=0$ and 
$n^1(L_*)=n^2(L_*)=1$ since $F$ is continuous in $x\in L_*$.
Then, $F/_S$ is $C^1$ by the last part of {\bf (A2)}-(3).
This finishes the proof.
\end{proof}

\begin{lemma}
\label{D(F)2}
Let $F:Dom(F) \subset \Sigma \to \Sigma$ be a triangular map 
satisfying {\bf (A1)-(A2)}. If $F$ has no periodic points 
and $Dom(F)=\Sigma \setminus L_0$, then $Dom(F)\setminus D(F)$ 
is open in $\Sigma$.
\end{lemma}

\begin{proof}
$Dom(F)$ is open in $\Sigma$
because $Dom(F)=\Sigma \setminus L_0$ and
$L_0$ is closed in $\Sigma$.
$Dom(F) \setminus D(F)$ is open in $Dom(F)$
by Lemma \ref{D(F)} because $F$ has no periodic
points and $ \partial ^v\Si \subset \Sigma\setminus L_0=Dom(F)$. Thus
$Dom(F) \setminus D(F)$ is open in $\Sigma$.
\end{proof}


Now, we need introduced some definitions for the next result. We say that 
$L$ is related with $L'$ $(L \sim L')$ if we have that $n(L), n(L') \geq 1$ 
and $L,L' \subset S_0(L) \cap S_0(L')$.

\begin{defi}
Let $L$ be a leaf of $\f$ and $L \in Dom(F)$. If $n(L) \geq 1$, we define the {\em leaf class} 
associated to the leaf $L$ as 
$$ \left\langle L\right\rangle = \left\{ L'\in Dom(F) \left| L' \sim L \right.\right\}.$$
If $n(L) =0$, we say that $\ce L \cd = \left\{L\right\}$.
\end{defi}

\begin{rk}
If $X$ is a vector field of codimension $1$, i.e., $dim (E^c)=2$, one has that 
$\ce L\cd=\left\{L\right\}$ and so $dim (\ce L\cd)=dim(\left\{L\right\})=s$ 
\end{rk}
 
\begin{lemma}
Let $L$ be a leaf in $D(F)$. Then $\ce L\cd$ is a $(s+1)$-submanifold 
(or $s$-submanifold if $\ce L\cd = \left\{L\right\}$) of $\Si$, and whose boundary 
belong to $\partial ^v \Si$.
\end{lemma}

\begin{proof}
By using the Lemma \ref{D(F)}, $D(F)$ is closed in $\Si \setminus L_0$. So, given $L \in Dom(F)$ 
by (A2) there is $S(L)=S(L)_0 \cup S(L)_1 \cup S(L)_2$. Then, $Cl( S(L)_0) \subset D(F)$, 
and as $D(F)$ is $\f$-saturated, one has that $Cl( S(L)_0) \setminus S(L)_0 = \partial S(L)_0$ 
are leaves. Thus, for each $L'\in \partial S(L)_0 \subset D(F)$, by using the Lemma \ref{D(F)1}, 
we obtain that $F(L') \subset \partial^v \Si$. Then, again by (A2) there exists 
$S(L')=S(L')_0 \cup S(L')_1 \cup S(L')_2$. In the same way, we proceeds analogously 
for $S(L')$. Since $\Si$ has finite diameter, we conclude that $\ce L\cd$ has boundary in 
$\partial ^v \Si$.
\end{proof}

We have that if $L \in D(F)$, then $F(L) \in \partial^v \Si$ (Lemma \ref{D(F)1}), and 
this motivate the following definition.

\begin{defi}
We define the {\em discontinuous class of leaves} by the set 
$$\ce D(F) \cd = \left\{ \ce L\cd \left| L \in D(F)\right.\right\}$$
\end{defi} 

\begin{defi}
\label{discrete}
A subset $B$ of $\Sigma$ is {\em ${\mathcal F}$-discrete}
if it corresponds to a set of leaves whose only points of 
accumulations are the leaves in $L_0$.
\end{defi}

\begin{lemma}
\label{tri-1}
If $F$ has no periodic points, then $\ce D(F) \cd $ is discrete.
\end{lemma}

\begin{proof}
By contradiction, we suppose that
$\ce D(F) \cd$ is not ${\mathcal F}$-discrete.
Then, there is an open neighborhood $U$
of $L_0$ in $\Sigma$ such that $ \ce D(F) \cd \setminus {\mathcal A}$
contains infinitely many classes $\ce L_n \cd$, where 
$${\mathcal A} = \left\{ \ce L\cd \in \ce D(F) \cd \left| \ce L \cd \cap U \neq \emptyset \right.\right\}.$$

By using Lemma \ref{D(F)}, $D(F)$ is closed in $Dom(F)=\Sigma \setminus L_0$, and so 
$\ce D(F) \cd$ is closed in $Dom(F) = \Si \setminus L_0$ too. Since $D(F) \setminus U$ 
is closed in $Dom(F) \setminus U$, one has that $\ce D(F) \cd \setminus {\mathcal A}$ 
is closed. As $U$ is an open neighborhood of $L_0$ and $Dom(F)=\Sigma\setminus L_0$
we obtain that $Dom(F) \setminus U$ is compact in $\Sigma$.
Henceforth $\ce D(F) \cd \setminus {\mathcal A}$ is compact.
So, without loss of generality, we can assume that $\ce L_n \cd$ converges to a class 
$\ce L_* \cd$ of $\ce D(F) \cd \setminus {\mathcal A}$. 
Clearly $ \ce L_* \cd \subset Dom(F)$.
Since $ \ce L_n \cd \subset D(F)$ we have
$F( \ce L_* \cd) \subset \partial^v \Si$ by Lemma \ref{D(F)1}.
It follows that $n(W_*)\geq 1$, for all $W_* \in \ce L_* \cd$.
We also have $n(W_*)\leq2k<\infty$ by Lemma
\ref{properties}-(2) since
$F$ has no periodic points and $\partial^v \Si \subset \Sigma\setminus L_0=Dom(F)$, 
for all $W_* \in \ce L_* \cd$.
By Definition \ref{n(L)} we have $f^{n(L_*)}(L_*)\subset \partial^v \Si \subset Dom(F)$.

By definition $\ce L_n \cd \cap \ce L_* \cd = \emptyset$ for all $n \in \nat$. 
Now, by using the property (A2), for each $L_n$ and for $L_*$ we can choose the following neighborhood 
associated to $\ce L_n \cd$ and $\ce L_* \cd$ as follows:
 
$$CS_n = \bigcup_{W \in \ce L_n \cd} S(W) \quad and \quad CS_* = \bigcup_{W_* \in \ce L* \cd} S(W_*).$$
Since $\ce L_n\cd$ is compact, one has that 
$$CS_n = \bigcup_{i=1}^k S(W_i) \quad and \quad CS_* = \bigcup_{i=1}^k S(W_{*,i}).$$

As $\ce L_n\cd \to \ce L_*\cd$ and $\ce L_n\cd \cap \ce L_*\cd=\emptyset$
we can assume $\ce L_n\cd \subset CS_*\setminus \ce L_*\cd$ for all $n$.
As $\ce L_n\cd \cap \ce L_*\cd=\emptyset$ for all $n$ we
can further assume that
$\ce L_n\cd \subset CS_{1,*}$ where $CS_{1,*}$ is one of
the (possibly equal) connected components
of $CS_{*}\setminus \ce L_* \cd$, i.e., $$CS_{1,*} = \bigcup_{i=1}^k S_1(W_i)$$.
	
As $F(S_1(L_n)) \subset \Sigma\setminus (\partial ^v \Si)$ for all $n \in \nat$
by {\bf (A2)}-(1) we conclude that $F(\ce L_n \cd )\subset \Sigma\setminus (\partial ^v \Si)$ for all $n$.
However, $F(\ce L_n \cd)\subset \partial ^v \Si$ by
Lemma \ref{D(F)1} since $L_n\subset D(F)$ a contradiction.
This proves the lemma.

\end{proof}
 
\bigskip

We need to extend some definitions for next lemmas and propositions. A {\em vertical band} in 
$\Si$ between two vertical $s$-surfaces $L, L'$ in the same component $\Si$ is nothing but a 
cylinder $H$ such that $L, L' \in \partial H$ and whose diameter is $l$, where $l$ represent the 
distance of $L$ to $L'$, i.e., $l= dist(L,L')$. Let us denote by $H(l,L')$ and $H\left[L,L'\right]$ 
the open and vertical band respectively.

Given a $u$-surface $c$, we say that $c$ is {\em tangent} 
to $C_{\alpha}$ if $Dc(t) \in C_{\alpha}(c(t))$ for all $t \in Dom(c) \subset \re^u$. 
A {\em $C_{\alpha}$-spine} of a vertical band $H(L,L')$(or $H \left[L,L'\right]$) is a 
$u$-surface $c \subset H(L,L')$ tangent to $C_{\alpha}$, such that 
$\partial c \subset \partial H(L,L')$ and $ int(c) \subset H(L,L')$.

\bigskip

\begin{lemma}
\label{tri-2}
Let $c \subset Dom(F) \setminus D(F)$ be and open $u$-surface transversal to $\f$. 
If there is $n \geq 1$ and open $C^1$ $u$-surface $c^*$ whose clausure $Cl(c^*) \subset c$ 
and such that $F^i(c^*) \subset (Dom(F) \setminus D(F)$ for all $0 \leq i\leq n-1$ and  
$F^n(c^*)$ covers $c$, then $F$ has a periodic point. 
\end{lemma}

\begin{proof}
We prove the lemma by contradiction. Then, we suppose that $F$ has no 
periodic point. So, by using the Lemma \ref{D(F)}, $Dom(F) \setminus D(F)$ 
is saturated and $F\left|_{Dom(F) \setminus D(F)}\right.$ is $C^1$.
Then, $c$ and $c^*$, projects (via $\f$) into two $u$-balls in $SL$ still denoted by 
$c$ and $c^*$ respectively. The assumptions imply that $f^i(c)$ is defined for all 
$0\leq i \leq n-1$ and $$Cl(c^*) \subset c \subset f^n(c^*).$$
Then, $f^n$ has a periodic point $L_{**}$. As $F^n(L_{**}) \subset f(L_{**}) = L_{**}$ and 
$F^n\left|_{L_{**}}\right.$ is continuous, then the {\em Brower fixed point Theorem} 
implies that $F^n$ has a fixed point. This fixed point represents a periodic point of $F$.
\end{proof}

\begin{lemma}
\label{tri-3}
$F$ carries a $u$-surface $c \subset Dom(F) \setminus D(F)$ tangent to $C_{\alpha}$ 
(with volume $V(c)$) into a $u$-surface tangent to $C_{\alpha}$ 
(with volume $\geq \lambda \cdot V(c)$).
\end{lemma}

\begin{proof}
See \cite{bm}.
\end{proof}

\begin{lemma}
\label{tri-4}
Suppose that $F$ has no periodic points. Let $L,L'$ be different leaves
in $D(F)$ such that the open vertical band $H(L,L')\subset Dom(F)\setminus D(F)$.
If $c$ is a $C_\alpha$-spine of $H(L,L')$,
then $F(Int(c))$ covers a vertical band
$H(W,W')$ with
$$ W,W'\subset \partial ^v \Si \cup {\mathcal V}. $$
\end{lemma}

\begin{proof}
By using Lemma \ref{D(F)} we have that
$F/_{H(L,L')}$ is $C^1$ because $H(L,L') \subset Dom(F) \setminus D(F)$. And
Lemma \ref{D(F)1} implies
\begin{equation}
\label{rastrojo}
F(L),F(L') \subset \partial ^v \Si
\end{equation}
because $L,L'\subset D(F)$.
Clearly $L,L'\subset Dom(F)$ and then $n(L),n(L')$ are defined.

By (\ref{rastrojo}) we have $n(L),n(L')\geq 1$.
Then, $1\leq n(L),n(L')<\infty$ by Lemma \ref{properties}-(1) since
$F$ has no periodic points and $ \partial ^v \Si \subset \Sigma\setminus L_0=Dom(F)$. 
By the same reason
$$F^{n(L)}(L),F^{n(L')}(L')\subset Dom(F).$$
If there exists a sequence converging to $L$ or $L'$, by using the Lemma \ref{remarkD(F)} 
exist the limit $F(L_n)_{n\in \nat}$ and this limit belong to $\partial ^v \Si \cup{\mathcal V}$. 
Let $W$ and $W'$ be these limits respectively. 

Now, let $c$ be a $C_\alpha$-spine of $H(L,L')$.
To fix ideas we assume $\partial c \subset \partial H(L,L')$, 
and this implies that there are $p,q \in \partial c$ such that 
$p \in L$ and $q \in L'$.
As $Int(c)\subset H(L,L')\subset Dom(F)\setminus D(F)$
we have that $F(Int(c))$ is defined. As $F/_{H(L,L')}$ is $C^1$ we
have that $F(Int(c))$ is a $u$-surface whose boundary containing points that belong 
in $\partial ^v \Si \cup{\mathcal V}$.

Clearly $W\neq W'$ because $F$ preserves ${\mathcal F}$.
Then, ${\mathcal F}_{F(Int(c))}=H(W,W')$ is an open vertical band, where 
$W, W' \subset \partial ^v \Si \cup{\mathcal V}$.
\end{proof}



\begin{defi}
Let $p$ be a point of $M$. We define the {\em radius of injectivity} $inj(p)$ 
at a point $p$ as the largest radius for which the exponential map at $p$ 
is a injective map, i.e.,
$$inj(p) = Sup \left\{ r>0 \left|  exp_p: B(0,r) \longrightarrow M \quad is \quad injective \right.\right\}.$$

Also, we say that the radius of injectivity  of the manifold $M$ is the infimum of the radius at all points, i.e,.

$$inj(M) = Inf \left\{ inj(p) \left| p \in M \right.\right\}.$$


\end{defi}


\begin{lemma}
\label{tri-5}
Suppose that $F$ has no periodic points.
For every open $u$-surface $c\subset Dom(F)\setminus
D(F)$ tangent to $C_\alpha$ there are an
open $u$-surface
$c^*\subset c$ and $n'(c)>0$ such that
$F^j (c^*)\subset Dom(F)\setminus D(F)$
for all $0\leq j \leq n'(c)-1$ and
$F^{n'(c)}(c^*)$ covers a vertical band
$H(W,W')$ with
$$
W,W'\subset \partial^v \Sigma \cup {\mathcal V} \cup {\mathcal L}.$$
\end{lemma}

\begin{proof}
Let $c\subset Dom(F)\setminus D(F)$ be
an $u$'surface tangent to $C_\alpha$.

In the same way of \cite{bm}, the proof is based on the following, 
that it will be modified for the higher dimensional case. . 
This claim will be proved by modifying the same arguments used in \cite{bm},\cite{gw} 
and a similar argument used by \cite{bpv} beside the radius of injectivity definition.

\begin{claim}
\label{chegou}
There are an open $u$-surface $c^{**}\subset c$ and
$n''(c)>0$ such that $F^j (c^{**})\subset Dom(F)\setminus D(F)$
for all $0\leq j \leq n''(c)-1$ and
$F^{n''(c)}(c^{**})$
covers an open vertical band
$$H(L,L')\subset Dom(F)\setminus \ce D(F) \cd ,$$
where $L,L'$ are different leaves
in $D(F)\cup L_0$.
\end{claim}

\begin{proof}
For every open $u$-surface $c' \subset Dom(F)\setminus D(F)$ tangent to
$C_\alpha$ we define
$$N(c')= \sup\left\{n\geq 1:
F^j(c')\subset Dom(F)\setminus \ce D(F) \cd, \forall 0 \leq j \leq n-1\right\}.$$
Note that $1 \leq N(c')  < \infty$
because $\lambda  > 1$ and
$\Sigma$ has finite diameter.
In addition, $F^{N(c')}(c')$ is a $u$-surface tangent
to $C_\alpha$ with
$$
F^{N(c')}(c') \cap (D(F)\cup L_0)
\neq \emptyset
$$
because $Dom(F)=\Sigma\setminus L_0$.

Define the number $\beta$ by
$$
\beta=(1/2)\cdot\lambda.
$$
Then, $\beta>1$ since $\lambda>2$.
Define $c_1=c$ and
$N_1=N(c_1)$.

Since $F^{N_1}(c_1)$ is a open $u$-surface, 
if $F^{N_1}(c_1)$ intersects $\ce D(F) \cd \cup L_0$
in a unique leaf class $\ce L_1 \cd$, then 
$F^{N_1}(c_1)\setminus \ce L_1 \cd$ has two connected components. 
In this case we define

\begin{itemize}
\item
$c_2^*=$ the connected component of
$F^{N_1}(c_1)\setminus \ce L_1 \cd $ that contain a ball whose 
radius of injectivity is greater or equal to any ball of the complement.
\item
$c_2=F^{-N_1}(c_2^*)$.
\end{itemize}

The following properties hold,

\begin{enumerate}
\item [1)]
$c_2\subset c_1$ and then $c_2$ is an open $u$-surface 
tangent to $C_\alpha$.
\item [2)]
$F^j(c_2)\subset Dom(F)\setminus \ce D(F) \cd $,
for all $0\leq j\leq N_1$.
\item [3)]
$V(F^{N_1}(c_2))\geq \beta\cdot V(c_1)$.
\end{enumerate}

In fact, the first property
follows because $F^{N_1}/_{{\mathcal F}_{c_2}}$
is injective and $C^1$.
The second one follows from
the definition of
$N_1=N(c_1)$ and from the fact
that $c_2^* = F^{N_1}(c_2)$ does not intersect
any leaf in
$ \ce D(F) \cd \cup L_0$.
The third one follows from
Lemma \ref{tri-3} because
$$
V( F^{N_1}(c_2)) =
V( c^*_2) \geq
(1/2) \cdot V( F^{N_1}(c_1) )\geq
$$
$$
\geq (1/2)\cdot \lambda^{N_1} V( c_1 ) \geq (1/2)\cdot\lambda V( c_1)
= \beta \cdot V( c_1 )
$$
since $\lambda>2$ and $N_1 \geq 1$.

Next we define $N_2=N(c_2)$.
The second property implies
$N_2>N_1$.
As before, if $F^{N_2}(c_2)$ intersects
$\ce D(F)\cd \cup L_0$
in a unique leaf class $\ce L_2 \cd$,
then $F^{N_2}(c_2)\setminus \ce L_2\cd $ has two connected components. 
In such a case we define analogously $c_3^*$
and also $c_3=F^{-N_2}(c_3^*)$.

As before
$$
V( F^{N_3}(c_3) )= V(c^*_3) \geq (1/2)\cdot V(F^{N_2}(c_2))
\geq (1/2)\cdot\lambda^{N_2-N_1} V(F^{N_1}(c_2)) \geq \beta^2V(c_1)
$$
because of the third property. So,

\begin{enumerate}
\item [1)]
$c_3\subset c_2$ and $c_3$ is an open $u$-surface tangent to
$C_\alpha$.
\item [2)]
$F^j(c_3)\subset Dom(F)\setminus \ce D(F)\cd $
for all $0\leq j\leq N_2$.
\item [3)]
$V(F^{N_2}(c_3))\geq \beta^2\cdot V( c_1)$.
\end{enumerate}

In this way we get a sequence
$N_1<N_2<N_3<\cdots<N_l<\cdots$ of positive integers
and a sequence $c_1,c_2,c_3,\cdots c_l,\cdots$
of open $u$-surfaces (in $c$)
such that the following properties hold
$\forall l \geq 1$

\begin{enumerate}
\item [1)]
$c_{l+1}\subset c_l$  and $c_{l+1}$ is
an open $u$-surface tangent to $C_\alpha$.
\item [2)]
$F^j(c_{l+1})\subset Dom(F)\setminus \ce D(F)\cd$
for all $0\leq j \leq N_l$.
\item [3)]
$V(F^{N_l}(c_{l+1}) ) \geq \beta^l \cdot V(c_1)$.
\end{enumerate}

The sequence $c_l$ must stop
by Property (3) since
$\Sigma$ has finite diameter.
So, {\em there is a
first integer $l_0$ such that
$F^{N(c_{l_0})}(c_{l_0})$ intersects $\ce D(F)\cd \cup L_0$
in two
different leaves class $\ce L\cd ,\ce L'\cd$}. Note that
these classes must be contained in the same component
of $\Sigma$ since $F^{N(c_{l_0})}(c_{l_0})$
is connected. Hence, we can suposse that the vertical band $H(L,L')$
bounded by $L,L'$ is well defined.
We can assume that $H(L,L')\subset
Dom(F)\setminus \ce D(F)\cd $ because
$\ce D(F)\cd $ is ${\mathcal F}$-discrete
by Lemma \ref{tri-1}.
Choosing $c^{**}=c_{l_0}$ and
$n''(c)=N_{l_0}$
we get the result.
\end{proof}

Now we finish the proof of
Lemma \ref{tri-5}.
Let $c^{**},n''(c)$ and $L,L'\subset D(F)\cup L_0$
be as in Claim
\ref{chegou}.
We have three possibilities:
$L,L'\subset D(F)$; 
$L\subset L_0$ and $L'\subset D(F)$;
$L\subset D(F)$ and $L'\subset L_0$.
We only consider the two first cases
since the later is similar
to the second one.

First we assume that $L,L'\subset D(F)$.
As
$F^{n''(c)}(c^{**})$ is tangent to $C_\alpha$, and
covers $H(L,L')$,
we can assume that $F^{n''(c)}(c^{**})$ itself
is a $C_\alpha$-spine of $H(L,L')$.
Then, applying Lemma \ref{tri-4}
to this spine, one gets that
$F^{n''(c)+1}(c^{**})$
covers a vertical band
$H(W,W')$ with
$$
W,W'\subset \partial ^v \Si \cup {\mathcal V}
$$
In this case
the choices $c^*=c^{**}$ and $n'(c)=n''(c)+1$
satisfy the conclusion of Lemma \ref{tri-5}.

Finally we assume that
$L\subset L_0$ and $L'\subset D(F)$.
As $L\subset L_0$ we have $L=L_{0i}$
for some $i=1,\cdots,k$.

On the one hand, $H(L_{0i},L')=H(L,L')\subset Dom(F)\setminus \ce D(F) \cd$
and then $\ce D(F)\cd \cap H(L_{0i},L')=\emptyset$.
So,
Lemma \ref{tri-1} implies that exists the limit $K_i$.
Consequently
$$
K_i\in
{\mathcal L}.
$$

On the other hand,
$F(L')\subset \partial^v \Si$
by Lemma \ref{D(F)1} since $L'\subset D(F)$.
It follows that $1\leq n(L')$ and also
$n(L')\leq 2k$ by Lemma \ref{properties}-(1)
since $F$ has no periodic
points and
$\partial^v \Si\subset \Sigma\setminus L_0=Dom(F)$.
Since $F^{n(L')}(L')\subset \partial^v \Si$
by the definition of $n(L')$
we obtain
$$
F^{n(L')}(L')\subset Dom(F).
$$

Then, Lemma \ref{remarkD(F)} applied to $L'$
implies that $lim_{L\rightarrow L'}f(L) = f^*(L')$ exists and satisfies
$$
f*(L')\subset (\partial^v \Si)
\cup
{\mathcal L}.
$$
But
$F(H(L_{0i},L'))$ (and so $F(F^{n''(c)}(c^{**}))$)
covers $H(K_i,f*(L'))$
since $H(L_{0i},L')\subset Dom(F)\setminus \ce D(F)\cd$.
Setting $W=K_i$ and $W'=f*(L')$ we
get
$$
W,W'\subset
(\partial^v \Si) \cup
{\mathcal V}
\cup {\mathcal L}.
$$
(Recall the definition of ${\mathcal V}$
in Definition \ref{limite-f})
Then,
$F(F^{n''(c)}(c^{**}))$ covers $(W,W')$ as in the statement.
Choosing
$c^*=c^{**}$
and $n'(c)=n''(c)+1$ we obtain
the result.
\end{proof}

\subsection{Proof of the Theorem \ref{op2}}
\label{to2}
Finally we prove Theorem \ref{op2}.
Let $F$ be a $\lambda$-hyperbolic
$n$-triangular map satisfying {\bf (A1)}-{\bf (A2)}
with $ \lambda  > 2$ and $Dom(F)=\Sigma \setminus L_0$.
We assume by contradiction that the following
property holds:

\begin{description}
\item[(P)]
$F$ has no periodic points.
\end{description}

Since $\partial ^v \Si \subset \Sigma\setminus L_0$
and $\Sigma\setminus L_0=Dom(F)$ we also have
$$
\partial ^v \Si \subset Dom(F).
$$
Then, the results in the previous subsections apply.
In particular, we have that
$Dom(F)\setminus D(F)$ is open in $\Sigma$
(by Lemma \ref{D(F)2}) and that
$\ce D(F) \cd$ is
${\mathcal F}$-discrete
(by Lemma \ref{tri-1}-(1)). All together imply that
$Dom(F)\setminus D(F)$ is open-dense
in $\Sigma$.

Now, let ${\mathcal B}$ be a family of open
vertical bands of the form $H(W,W')$ with
$$
W,W'\subset \partial ^v \Si \cup
{\mathcal V}\cup{\mathcal L}.
$$
It is clear that
${\mathcal B}=\{B_1,\cdots, B_m\}$
is a finite set.
In ${\mathcal B}$ we define the relation
$B\leq B'$ if and only if
there are an open $u$-surfaces $c\subset B$
tangent to $C_\alpha$ with closure $\text{CL}(c)\subset Dom(F)\setminus \ce D(F)\cd$, an open $u$-surface
$c^{*}\subset c$ and
$n> 0$ such that
$$
F^j(c^{*})\subset Dom(F)\setminus \ce D(F) \cd,
\,\,\,\,\,\,\,\,\,\,\,\,\,\,\,\,\,\,
\forall 0\leq j \leq n-1,
$$
and
$F^{n}(c^{*})$ covers
$B'$.

As $Dom(F)\setminus D(F)$ is open-dense
in $\Sigma$, and the bands in ${\mathcal B}$
are open, we can use
Lemma \ref{tri-5} to prove that
for every $B\in {\mathcal B}$
there is $B'\in {\mathcal B}$ such that
$B\leq B'$.
Then, we can construct
a chain
$$
B_{j_1}\leq B_{j_1}\leq B_{j_2}\leq \cdots,
$$
with $j_i\in \{1,\cdots, m\}$ ($\forall i$) and
$j_1=1$.
As ${\mathcal B}$ is finite
it would exist
a closed sub-chain
$$
B_{j_i}\leq B_{j_{i+1}}\leq
\cdots \leq B_{j_{i+s}}\leq B_{j_i}.
$$
Hence there a positive integer $n$ such that
$F^n(B_{j_i})$ covers $B_{j_i}$.
Applying Lemma \ref{tri-2} to suitable
$u$-surfaces $c^*\subset Cl(c^*)\subset c\subset B_{j_i}$ we obtain that
{\em $F$ has a periodic point}.
This contradicts {\bf (P)} and the proof follows.

\subsection{Proof of the Main Theorem}

\begin{proof}
Let $X$ be a sectional Anosov flow on a compact $n$-manifold $M$. 
We prove the Theorem by contradiction, i.e, to prove that $M(X)$ has a periodic orbit
we assume that this is not so. Fix $\lambda > 2$. As 
there is a singular cross-section $\Si$ close to $M(X)$ such that if $F$ is the return map of
the refinement $\Si(\delta)$ of $\Si$, 
then there is $\delta > 0$ such that $F$ is a
$\lambda$-hyperbolic $n$-triangular map with $Dom(F) = \Si \ L_0$. 
We also have that $F$ satisfies (A1) and (A2) in Subsection \ref{H} 
since $\Si$ is close to
$\Lambda$. Then, $F$ has a periodic point by Theorem \ref{op2} since $\lambda> 2$. This periodic point
belongs to a periodic orbit of $X_t$ which in turns belongs to $M(X)$ since it is maximal invariant.
Consequently $X_t$ has a periodic orbit in $M(X)$, a contradiction. This contradiction
proves the result.
\end{proof}

\medskip 

\flushleft
A. M. L\'opez B\\
Instituto de Matem\'atica, Universidade Federal do Rio de Janeiro\\
Rio de Janeiro, Brazil\\
E-mail: barragan@im.ufrj.br

\end{document}